\documentclass{amsart}
\newfont{\cyr}{wncyr10} 

\newcommand{\map}[1]{\;\xrightarrow{#1}\;}

\newcommand{\Sha}{\mbox{\cyr Sh}}

\newcommand{\Gal}{\mathrm{Gal}}
\newcommand{\Hom}{\mathrm{Hom}}
\newcommand{\Q}{\mathbf Q}
\newcommand{\Z}{\mathbf Z}
\newcommand{\cO}{\mathcal O}
\newcommand{\gm}{\mathfrak m}

\newcommand{\inert}{\mathcal{I}}
\newcommand{\Sel}{\mathrm{Sel}}
\newcommand{\ev}{\mathrm{ev}}

\newcommand{\taut}{\omega_\mathrm{taut}}

\newcommand{\Mod}{\mathbf{Mod}}
\newcommand{\sel}{\mathcal F}
\newcommand{\msel}{{/\sel}}
\newcommand{\seldual}{{\sel^\vee}}

\newcommand{\altsel}{\mathcal G}
\newcommand{\rel}{\mathrm{rel}}
\newcommand{\loc}{\mathrm{loc}}
\newcommand{\res}{\mathrm{res}}

\newcommand{\rank}{\mathrm{rank}}
\newcommand{\corank}{\mathrm{corank}}
\newcommand{\Iw}{\mathrm{Iw}}
\newcommand{\inv}{\mathrm{inv}}

\newcommand{\ord}{\mathrm{ord}}
\newcommand{\Der}{\mathrm{Der}}
\newcommand{\unr}{\mathrm{unr}}

\newcommand{\Adual}{A^\vee}
\newcommand{\bS}{\mathbf{S}}
\newcommand{\Fil}{\mathrm{Fil}}
\newcommand{\gr}{\mathrm{gr}}
\newcommand{\pl}{\mathcal{L}}

\newcommand{\mil}{\lim\limits_\leftarrow}
\newcommand{\dlim}{\lim\limits_\rightarrow}
\newcommand{\iso}{\cong}

\input xy
\xyoption{all}

\begin{document}
\title{Derived P-adic heights and p-adic L-functions}
\author{Benjamin Howard}
\address{Department of Mathematics\\ Stanford University\\ Stanford, CA\\
94305}
\curraddr{Department of Mathematics\\ Harvard University\\ Cambridge, MA\\
02138}
\thanks{This research was partially conducted by the author for the Clay
Mathematics Institute.}

\begin{abstract}
If $E$ is an elliptic curve defined over a number field and $p$
is a prime of good ordinary reduction for $E$, a theorem of 
Rubin \cite{rubin-height}, Thm. 1, 
relates the $p$-adic height pairing on the $p$-power Selmer
group of $E$ to the first derivative of 
a cohomologically defined $p$-adic $L$-function attached to $E$.  
Bertolini and Darmon \cite{derived-heights} have defined 
a sequence of ``derived'' $p$-adic heights. In this paper
we give an alternative definition of the $p$-adic height pairing and 
prove a generalization
of Rubin's result, relating the derived heights to higher derivatives 
of $p$-adic $L$-functions.  We also relate degeneracies in the derived
heights to the failure of the Selmer group of  $E$ over a 
$\mathbf{Z}_p$-extension to be ``semi-simple'' as an Iwasawa module,
generalizing results of Perrin-Riou \cite{perrin-riou}.
\end{abstract}

\maketitle

\theoremstyle{plain}
\newtheorem{Thm}{Theorem}[section]
\newtheorem{Prop}[Thm]{Proposition}
\newtheorem{Lem}[Thm]{Lemma}
\newtheorem{Cor}[Thm]{Corollary}
\newtheorem{Conj}[Thm]{Conjecture}

\theoremstyle{definition}
\newtheorem{Def}[Thm]{Definition}

\theoremstyle{remark}
\newtheorem{Rem}[Thm]{Remark}
\newtheorem{Ques}[Thm]{Question}

\renewcommand{\labelenumi}{(\alph{enumi})}
\setcounter{section}{-1}

\section{Introduction and Notation}

Fix forever a rational prime $p>2$. By a coefficient ring we 
mean a commutative ring which is complete, Noetherian, and 
local with residue field of characteristic $p$.
Fix  a finite set of places $\Sigma$ of a number field 
$F$ containing all archimedean places and all primes 
above $p$, and let $G_\Sigma=\Gal(F_\Sigma/F)$ where $F_\Sigma$ is the 
maximal extension of $F$ unramified outside of $\Sigma$.  
For any coefficient ring $R$, denote by $\Mod_\Sigma(R)$
the category of free $R$-modules of finite type equipped with continuous,
$R$-linear actions of $G_\Sigma$.
The notation $M(k)$ for any $G_\Sigma$-module $M$
means Tate twist, as usual.  

Throughout this article we work with a fixed coefficient ring 
$\cO$, which is assumed  to be 
topologically discrete (at least until Section \ref{semi-simple}), 
for example $\cO=\Z/p^k\Z$.
With $F$ as above let $F_\infty/F$ be a $\Z_p$-extension.
If $F_n\subset F_\infty$ is the unique subfield
with $[F_n:F]=p^n$, we define
$$\Gamma_n=\Gal(F_n/F)\hspace{1cm}\Gamma=\Gal(F_\infty/F)\hspace{1cm}
\Lambda_n=\cO[\Gamma_n]\hspace{1cm}\Lambda=\cO[[\Gamma]],$$ and denote
by $J$ the augmentation ideal of $\Lambda$.
Let $\gamma\in\Gamma$ be a topological generator and denote by $\iota$
the involution of $\Lambda$ induced by 
$\gamma\mapsto \gamma^{-1}$.  We will also view $\iota$ as a functor
$M\mapsto M^\iota$ from the category of $\Lambda$-modules to itself,
where the underlying group of $M^\iota$ is the same as that of $M$
but with $\Lambda$ acting through 
$\Lambda\map{\iota}\Lambda\map{}\mathrm{End}_\cO(M)$.

In Section \ref{construction} we consider two objects $S$, $T$
of $\Mod_\Sigma(\cO)$ which are assumed to be in Cartier duality: i.e.
we assume that there exists a perfect $\cO$-bilinear, $G_\Sigma$-equivariant
pairing $S\times T\map{}\cO(1).$  For such objects we define
generalized Selmer groups 
$$H^1_\sel(F,S_\infty)\subset \dlim H^1(F_n,S)\hspace{1cm}
H^1_\altsel(F,T_\infty)\subset \dlim H^1(F_n,T)$$
and show that there is a canonical (up to sign) height pairing
$$H^1_\sel(F,S_\infty)\times H^1_\altsel(F,T_\infty)\map{}
J/J^2$$ whose kernels on either side are the submodules of \emph{universal
norms} in the sense of Definition \ref{norm def}.

We continue to work in great generality in Section \ref{derivative}, 
using the construction of Section \ref{construction} to define 
derived height pairings similar to those of Bertolini and Darmon.  
A theorem of Rubin \cite{rubin-height}, Thm. 1,
relates the $p$-adic height pairing to
the special values of the first derivatives of certain (cohomologically
defined) $p$-adic $L$-functions, and we prove similar formulas relating 
the derived heights to special values of higher derivatives.

In Section \ref{abelian varieties} we consider the special case where
$S$ and $T$ arise from torsion points on an abelian variety with
good, ordinary reduction at primes of $F$ above $p$, and show that
our Selmer groups agree with the usual ones (up to a controlled error).

In Section \ref{semi-simple} we continue to work with torsion points
on an abelian variety $A$, and explain how degeneracies in the
derived heights are reflected in the structure of the 
Selmer group of $A$ over $F_\infty$
as an Iwasawa module, generalizing work of Perrin-Riou
\cite{perrin-riou}.

The reader is encouraged to begin by reading the results of 
Section \ref{semi-simple}, in particular Theorems \ref{vector height} and 
\ref{height formula}, and Corollary \ref{semi-simplicity}.


\section{Construction of the $p$-adic height pairing}
\label{construction}
  
We wish to work with a very general notion of Selmer group, borrowing 
some notation and conventions from \cite{mazur-rubin}.

\begin{Def}
Suppose $R$ is a coefficient ring and $M$ is topological $R$-module
equipped with a continuous $R$-linear action of $G_\Sigma$. 
A \emph{Selmer structure}, $\sel$, on $M$ is a choice of $R$-submodule
$$H^1_\sel(F_v,M)\subset H^1(F_v,M)$$
for every $v\in\Sigma$.  Given a Selmer structure, the associated
\emph{Selmer module} $H^1_\sel(F,M)$   is defined to be the kernel of
$$H^1(G_\Sigma,M)\map{}\bigoplus_{v\in\Sigma}
H^1(F_v,M)/H^1_\sel(F_v,M).$$ 
\end{Def}

 If we are given a Selmer structure $\sel$ on 
$M$ and a surjection $M\map{}M'$, we obtain a Selmer structure on $M'$
(still denoted $\sel$)
by taking the local condition at any $v\in\Sigma$ to be the image of
$$H^1_\sel(F_v,M)\map{}H^1(F_v,M').$$
If instead we are given an injection $M'\map{}M$ then we define a Selmer
structure on $M'$ by taking the preimage of $H^1_\sel(F_v,M)$
under $$H^1(F_v,M')\map{}H^1(F_v,M).$$  We will refer to this as 
\emph{propagation} of Selmer structures.

If $(T,\pi)$ belongs to $\Mod_\Sigma(\cO)$, then induction from
$\Gal(F_\Sigma/F_n)$ to $\Gal(F_\Sigma/F)$ defines a module
$(T_n,\pi_n)$ in 
$\Mod_\Sigma(\Lambda_n)$. Explicitly, 
$$T_n=\{f:G_\Sigma\map{}T\mid f(gx)=\pi(g) f(x)\ \forall g\in 
\Gal(F_\Sigma/F_n)\}.$$
Here $G_\Sigma$ acts by $(\pi_n(g)f)(x)=f(xg)$ and $\Lambda_n$
acts by $(\gamma f)(x)=\pi(\tilde{\gamma})f(\tilde{\gamma}^{-1}x)$
where $\tilde{\gamma}$ is any lift of $\gamma\in\Gamma_n$ to $G_\Sigma$.
Define $$T_\infty=\dlim T_n\hspace{1cm}T_\Iw=\mil T_n$$
where the direct limit is with respect to the natural inclusions 
(restriction) and
the inverse limit is with respect to the norm operators in $\Lambda_n$
(corestriction). Let $\ev:T_\infty\map{}T$ be evaluation at 
the identity element of $G_\Sigma$. By Shapiro's
lemma the composition
$$H^i(F,T_\infty)\map{\res}H^i(F_\infty,T_\infty)\map{\ev}H^i(F_\infty,T)$$
is an isomorphism, and similarly we have
$$H^i(F,T_\Iw)\iso \mil H^i(F_n,T).$$

An element $\lambda\in\Lambda$ is said
to be \emph{distinguished} if $\lambda\not\in \gm\Lambda$, where
$\gm$ is the maximal ideal of $\cO$.

\begin{Lem}\label{coterminal}
If $g_n$ is any sequence of elements of $\Lambda$ which
converges to zero and $f\in\Lambda$ is distinguished, then $f\mid g_n$
for $n\gg 0$.
\end{Lem}
\begin{proof} Identify $\Lambda$ with a power series ring.  The map taking
$\lambda\in\Lambda$ to its remainder upon division by $f$ is continuous,
and so if we write $g_n=q_nf+r_n$ with degree of $r_n$ less than that
of $f$, we have $r_n\to 0$ with the $r_n$ running through a discrete set
(recall $\cO$ is assumed discrete). \end{proof}

If we fix a topological generator $\gamma\in\Gamma$ and let
$g_n=\frac{\gamma^{p^{n}}-1}{\gamma-1}$ be the norm element in $\Lambda_n$,
the above lemma shows that $m\in M$ is divisible by every distinguished
$f\in\Lambda$ if and only if $m\in g_n M$ for every $n$. This motivates
the following

\begin{Def}\label{norm def}
If $M$ is a $\Lambda$-module, we say that $m\in M$ is a 
\emph{universal norm} if $m\in f M$ for every distinguished $f\in\Lambda$.
\end{Def}

We denote by $\taut:G_\Sigma\map{}\Gamma\map{}\Lambda^\times$ 
the tautological character, and let $\Lambda\{k\}$ denote the
ring $\Lambda$, viewed as a module over itself,
on which $G_\Sigma$ acts through $\taut^k$.  The unadorned symbol 
$\Lambda$ is always interpreted as $\Lambda\{0\}$, i.e. with trivial 
Galois action. For any object $M$ on which $\Lambda$ and $G_\Sigma$ act
we let $M\{k\}=M\otimes_\Lambda\Lambda\{k\}$, and we regard the underlying
$\Lambda$-modules of $M$ and $M\{k\}$ as being identified
via $m\mapsto m\otimes 1$.

\begin{Lem}\label{induced}
For $f\in T_n$ and $\gamma\in\Gamma_n$, set 
$\mu_f(\gamma)=\ev(\gamma^{-1}f)\in T.$ The map $f\mapsto\sum_\gamma
\mu_f(\gamma)\otimes\gamma$ defines an isomorphism 
in $\Mod_\Sigma(\Lambda_n)$ $$T_n\iso T\otimes_\cO\Lambda_n\{-1\}.$$
In the limit this defines an isomorphism $T_\Iw\iso T\otimes_\cO\Lambda\{-1\}.$
\end{Lem}
\begin{proof} This amounts to verifying the relations
$$ \mu_{\gamma_0f}(\gamma)=\mu_f(\gamma_0^{-1}\gamma)
\hspace{1cm}\mu_{\pi_n(g)f}(\gamma)=
\pi(g)\big(\mu_f(\taut(g)\gamma )\big)$$ which are elementary.
\end{proof}

Let $K$ denote the ring obtained by localizing $\Lambda$ at the prime
generated by the maximal ideal of $\cO$, i.e. inverting all distinguished 
elements.  Applying Lemma \ref{coterminal} one
may easily check that $K$ is noncanonically isomorphic to the ring
of Laurent series with non-essential singularities 
in one variable over $\cO$.  Define $P$ (the module of poles)
by exactness of the sequence 
$$0\map{}\Lambda\map{}K\map{}P\map{}0.$$ For any
$T$ in $\Mod_\Sigma(\cO)$ we tensor the above sequence (over $\Lambda$)
with $T_\Iw$ to obtain
$$0\map{}T_\Iw\map{}T_K\map{}T_P\map{}0.$$
By Lemma \ref{induced} there are canonical identifications
of $\Lambda$-modules and $G_\Sigma$-modules
 $\cO_\Iw\iso \Lambda\{-1\}$, $\cO_K\iso K\{-1\}$, and $\cO_P\iso P\{-1\}$.

\begin{Lem}\label{polar isomorphism}
A choice of topological generator of $\Gamma$ determines an
isomorphism $$T_P\iso T_\infty.$$  If the map associated to the generator
$\gamma$ is denoted $\eta_\gamma$ and $u\in\Z_p^\times$ then
$\eta_{\gamma^u}=u\cdot \eta_\gamma.$
\end{Lem}
\begin{proof} The map $\eta_\gamma$ is described as follows.  Any element of
$T_P[\gamma^{p^n}-1]$ may be written as $\tau\otimes(\gamma^{p^n}-1)^{-1}$
with $\tau\in T_\Iw$.  Such an element is sent by $\eta_\gamma$ to
the image of $\tau$ in $T_n\iso T_\infty[\gamma^{p^n}-1]$.
The relation between $\eta_\gamma$ and $\eta_{\gamma^u}$ follows from 
$$\frac{\gamma^{up^n}-1}{\gamma^{p^n}-1}\equiv 
u\pmod{(\gamma^{p^n}-1)\Lambda}.$$
\end{proof}

\begin{Lem}\label{convolution}
Suppose $e:S\times T\map{}\cO(1)$ is an $\cO$-bilinear, $G_\Sigma$-equivariant,
 perfect pairing  of objects in $\Mod_\Sigma(\cO)$.
There is an induced $G_\Sigma$-equivariant and  perfect pairing
$$e_n:S_n\times T_n\map{}\Lambda_n(1)$$  which satisfies
 $$e_n(\lambda s,t)=\lambda e_n(s,t)=e_n(s,\lambda^\iota t).$$
One may pass to the limit and then tensor with $K$ to obtain perfect 
pairings
$$e_\Iw:S_\Iw\times T_\Iw\map{}\Lambda(1)\hspace{1cm}
e_K:S_K\times T_K\map{}K(1).$$

 If $S=T$ and the pairing
on $S\times T$ is symmetric (resp. alternating) then the induced
pairings satisfy  $e_\bullet(s,t)=e_\bullet(t,s)^\iota$ (resp.
$e_\bullet(s,t)=-e_\bullet(t,s)^\iota$).
\end{Lem}
\begin{proof} Using the identifications of Lemma \ref{induced}, the pairing 
is defined by $e_n(s,t)=\sum_{\gamma\in\Gamma_n} \mu(\gamma)\otimes\gamma$
where
$\mu(\gamma)= \sum_{x\in\Gamma_n}e(\mu_s(x),\mu_t(x\gamma^{-1})).$
It is elementary to check that the stated properties hold.
\end{proof}

\begin{Lem}\label{unramified propagation}
Let $T$ be an object of $\Mod_\Sigma(\cO)$, and let 
$v$ be a prime of $F$ not dividing $p$.  The map
$$H^1_\unr(F_v,T_K)\map{}H^1_\unr(F_v,T_P)$$ is surjective.  
If $v$ is finitely decomposed in $F_\infty$, then
$H^1_\unr(F_v,T_P)=0$.
\end{Lem}
\begin{proof} Let $\inert\subset\Gal(\bar{F}_v/F_v)$ be the inertia
group of $v$. Since $T_K=T\otimes_{\Z_p} 
K\{-1\}$ and the restriction of $\taut$ to $\inert$ is trivial, 
$(T_K)^\inert=T^\inert\otimes K.$
Similarly, $(T_P)^\inert=T^\inert\otimes P$, and so $(T_K)^{\inert}$
surjects onto $(T_P)^{\inert}$.  Using the fact that 
$\Gal(F_v^\unr/F_v)$ has cohomological dimension one we deduce that the map
$$H^1(F_v^\unr/F_v,(T_K)^{\inert})\map{}
H^1(F_v^\unr/F_v,(T_P)^{\inert})$$
is surjective, proving the first claim.

If $v$ is finitely decomposed in $F_\infty$, then Lemma 
\ref{polar isomorphism} and Shapiro's lemma allow us to identify
$$H^1_\unr(F_v,T_P)\iso H^1_\unr(F_v,T_\infty)\iso 
\bigoplus_w H^1_\unr(F_{\infty,w},T)$$
where the sum is over places of $F_\infty$ above $v$. 
The pro-$p$-part of $\Gal(F_v^\unr/F_{\infty,w})$ is trivial, and so the 
right hand side is zero.
 \end{proof}

The construction of the pairing of the following theorem, as well
as the verification of its properties, is a modification
of the construction of the Cassels-Tate pairing as described in \cite{flach}.

\begin{Thm}\label{first pairing}
Suppose $S$ and $T$ are objects in $\Mod_\Sigma(\cO)$ and that
there is a perfect $G_\Sigma$-equivariant  pairing
$S\times T\map{}\cO(1).$ Suppose further  that we are given Selmer structures
$\sel$ and $\altsel$ on $S_K$ and $T_K$, respectively, which are everywhere
exact orthogonal complements under the pairing
$$S_K\times T_K\map{}K(1).$$  Then there is a canonical pairing
$$[\ ,\ ]:H^1_\sel(F,S_P)\times H^1_\altsel(F,T_P)\map{}P$$
whose kernels on the left and right are the images of $H^1_\sel(F,S_K)$
and $H^1_\altsel(F,T_K)$, and these images are exactly the submodules
of universal norms. 
This pairing satisfies $[\lambda s,t]=
\lambda[s,t]=[s,\lambda^\iota t]$. 
If $S=T$, $\sel=\altsel$, and the pairing on $T\times T$
is symmetric (resp. alternating) then we also have $[s,t]=[t,s]^\iota$
(resp. $[s,t]=-[t,s]^\iota$).  
\end{Thm}
\begin{proof} Given cocycles $s$ and $t$
representing classes in $H^1_\sel(F,S_P)$ and $H^1_\altsel(F,T_P)$, 
respectively, choose cochains
$$\tilde{s}\in C^1(G_\Sigma,S_K)\hspace{2cm}\tilde{t}\in C^1(G_\Sigma,T_K)$$
whose images under the maps induced by $S_K\map{}S_P$ and $T_K\map{}T_P$
are $s$ and $t$.  By the definition of propagation of Selmer structures, 
there are exact sequences
\begin{eqnarray}\label{propogated exact}
H^1_\sel(F_v,S_\Iw)\map{}H^1_\sel(F_v,S_K)\map{}H^1_\sel(F_v,S_P)\map{}0\\
H^1_\altsel(F_v,T_\Iw)\map{}H^1_\altsel(F_v,T_K)\map{}
H^1_\altsel(F_v,T_P)\map{}0\nonumber
\end{eqnarray} at every $v\in\Sigma$, and so we may choose 
semi-local classes
$$\tilde{s}_\Sigma\in\bigoplus_{v\in\Sigma}H^1_\sel(F_v,S_K)\hspace{1cm}
\tilde{t}_\Sigma\in\bigoplus_{v\in\Sigma}H^1_\altsel(F_v,T_K)$$
which reduce to the semi-localizations $$\loc_\Sigma(s)\in
\bigoplus_{v\in\Sigma}H^1_\sel(F_v,S_P)\hspace{1cm}
\loc_\Sigma(t)\in\bigoplus_{v\in\Sigma}H^1_\altsel(F_v,T_P).$$

From the fact that $s$ and $t$ are cocycles it follows that the image
of $d\tilde{s}\cup d\tilde{t}$ in  $C^4(G_\Sigma,P(1))$ is trivial, and
so also are the images of $d(d\tilde{s}\cup\tilde{t})=d\tilde{s}\cup
d\tilde{t}=d(\tilde{s}\cup d\tilde{t})$.  Using $H^3(G_\Sigma,P(1))=0$
we may therefore choose $\epsilon_0,\epsilon_1\in C^2(G_\Sigma,P(1))$ 
such that $$d\epsilon_0=d\tilde{s}\cup\tilde{t}
\hspace{2cm} d\epsilon_1=\tilde{s}\cup d\tilde{t}$$ in $Z^3(G_\Sigma,P(1))$.
Writing $\inv_\Sigma:\bigoplus_{v\in\Sigma}
H^2(F_v,P(1))\map{}P$ for the sum of the local invariants, we now define 
\begin{equation}\label{first pairing definition}
[s,t]=\inv_\Sigma\big(\loc_\Sigma(\tilde{s})\cup 
\tilde{t}_\Sigma-\loc_\Sigma
(\epsilon_0) \big)=-\inv_\Sigma\big(\tilde{s}_\Sigma\cup\loc_\Sigma(\tilde{t})
+\loc_\Sigma(\epsilon_1) \big)\end{equation}
 where the second equality follows from 
$(\loc_\Sigma(\tilde{s})-\tilde{s}_\Sigma)\cup 
(\loc_\Sigma(\tilde{t})-\tilde{t}_\Sigma)=0$ in $\bigoplus_{v\in\Sigma}
H^1(F_v,P(1))$ and the reciprocity law of class field theory.
It is elementary to check that this is independent of the choices
made (or see Flach's paper for essentially the same calculations).
Furthermore, it is clear from the construction that the kernels
on either side contain the images of $H^1_\sel(F,S_K)$ and 
$H^1_\altsel(F,T_K)$.

For a $\Lambda$-module $M$, we write $M^\vee=\Hom_\Lambda(M^\iota,P)$.  
If $M$ is a topological group on which $G_\Sigma$ acts continuously we define
$$\Sha^i(F,M)=\mathrm{ker}\big(H^i(G_\Sigma,M)\map{}\prod_{v\in\Sigma}
H^i(F_v,M)\big).$$  The Poitou-Tate nine-term exact sequence
provides a perfect pairing 
$$\Sha^2(F,S_\Iw)\times\Sha^1(F,T_P)\map{}P$$
which defines the right vertical arrow in the exact and commutative
diagram $$\xymatrix{
0\ar[r]&H^1_\sel(F,S_P)^0_{\ /K}\ar[r]&
H^1_\sel(F,S_P)_{/K}\ar[r]\ar[d]&{\Sha^2(F,S_\Iw)}\ar[d]\\
 &    &H^1_\altsel(F,T_P)^\vee\ar[r]  &{\Sha^1(F,T_P)^\vee}.}$$
Here $H^1_\sel(F,S_P)^0$ denotes $H^1_\sel(F,S_P)$ intersected
with the image of $H^1(F,S_K)$ in $H^1(F,S_P)$ and the subscript
$/K$ indicates quotient by the image of $H^1_\sel(F,S_K)$ in $H^1_\sel
(F,S_P)$. The top row is extracted from the cohomology of 
$$0\map{}S_\Iw\map{}S_K\map{}S_P\map{}0$$
using exactness of (\ref{propogated exact}). 

The left vertical arrow
is induced by the pairing of the theorem, and a diagram chase shows that
to check injectivity of this arrow it suffices to show that
$H^1_\sel(F,S_P)^0_{\ /K}$ injects into $H^1_\altsel(F,T_P)^\vee$.
In other words, if $s\in H^1_\sel(F,S_P)$ is in the kernel on the left
then we are free to assume that $\tilde{s}$ is chosen to be a cocycle and that 
$\epsilon_0=0$.  Then we have $\inv_\Sigma\big(\loc_\Sigma(\tilde{s})\cup
\tilde{t}_\Sigma\big)=0$ for every $\tilde{t}_\Sigma\in
\bigoplus_{v\in\Sigma}H^1_\altsel(F_v,T_K)$ whose image in 
$\bigoplus_{v\in\Sigma}H^1_\altsel(F_v,T_P)$ comes from a global 
$t\in H^1_\altsel(F,T_P)$. Denote by $c$ the image of $\loc_\Sigma(\tilde{s})$
in $\bigoplus_{v\in\Sigma}H^1(F_v,S_K)/H^1_\sel(F_v,S_K)$.  It follows
from the exactness of (\ref{propogated exact}) that there is a class
$$d\in\bigoplus_{v\in\Sigma}H^1(F_v,S_\Iw)/H^1_\sel(F_v,S_\Iw)$$
whose image in $\bigoplus_{v\in\Sigma}H^1(F_v,S_K)/H^1_\sel(F_v,S_K)$
is equal to $c$.  Then
\begin{eqnarray*} \inv_\Sigma(d\cup \loc_\Sigma(t))&=&
\inv_\Sigma(c\cup \tilde{t}_\Sigma)\\
&=&\inv_\Sigma(\loc_\Sigma(\tilde{s})\cup\tilde{t}_\Sigma)\\
&=&0\end{eqnarray*} in $P$
for every $t\in H^1_\altsel(F,T_P)$. It follows from Poitou-Tate
global duality that $d$ is the image of a global class 
$\delta\in H^1(G_\Sigma,S_\Iw)$, and that $\tilde{s}-\delta\in H^1_\sel(F,S_K)$
reduces to $s\in H^1_\sel(F,S_P)$.  This and a similar argument with
the roles of $S$ and $T$ reversed show that the kernels on the left and
right are contained in the images of 
$$H^1_\sel(F,S_K)\map{}H^1_\sel(F,S_P)\hspace{1cm}
H^1_\altsel(F,T_K)\map{}H^1_\altsel(F,T_P)$$ respectively.
The image of $H^1_\sel(F,S_K)$ is clearly contained in the universal norms
which are clearly contained in the left kernel,
 and similarly for $T$, and so the kernels
on either sides are exactly the universal norms.

The final claim regarding the case $S=T$ follows from the 
two descriptions of the pairing in (\ref{first pairing definition})
and the skew-symmetry of the cup product.  Details can be found in Flach's
paper.
\end{proof}

\begin{Lem}\label{evaluation}
Let $$\cO_\infty\{1\}=\dlim \cO_n\{1\}\iso\dlim \Lambda_n\{0\},$$ so that 
$\cO_\infty\{1\}$ is just $\cO_\infty$ with $G_\Sigma$ acting trivially.
For any $\Lambda$-module of finite type, $A$, the map $\ev:\cO_\infty\{1\}
\map{}\cO$
induces a canonical isomorphism in $\Mod_\Sigma(\cO)$
$$\Hom_\Lambda(A,\cO_\infty\{1\})\iso \Hom_\cO(A,\cO).$$
\end{Lem}
\begin{proof} This is a special case of Frobenius reciprocity.  The inverse map
may be described explicitly as follows: for $\phi\in\Hom_\cO(A,\cO)$
we must have $\phi((\gamma^{p^n}-1)A)=0$ for $n\gg 0$, by discreteness
of $\cO$ and continuity of $\phi$.  Taking $n$ very large we then define
$\Phi\in \Hom_\Lambda(A,\cO_n\{1\})$ by $\Phi(a)(g)=\phi(\taut(g)\cdot a)$
for $g\in G_\Sigma$.
\end{proof}

Keep the assumptions of Theorem \ref{first pairing}.
Fixing a topological generator $\gamma$ of $\Gamma$,  the isomorphisms
$\eta_\gamma$ of Lemma \ref{polar isomorphism} determine Selmer 
structures, still denoted by $\sel$ and $\altsel$, on $S_\infty$ and
$T_\infty$, and these Selmer structures do not depend on the choice of
$\gamma$.  The composition \begin{equation}\label{evaluation composition}
P\iso \cO_P\{1\}\map{\eta_\gamma}\cO_\infty\{1\}\map{\ev}\cO\end{equation} 
allows us
to construct from the pairing of Theorem \ref{first pairing} a pairing
$$h_\gamma:H^1_\sel(F,S_\infty)\times H^1_\altsel(F,T_\infty)
\map{}\cO$$ whose kernel (by Lemma \ref{evaluation}) on either side
is exactly the submodule of universal norms.

\begin{Lem}\label{pre-height lemma}
The above pairing satisfies 
\begin{enumerate}
\item $h_\gamma(\lambda s,t)=h_\gamma(s,\lambda^\iota t)$
\item for $u\in\Z_p^\times$, $h_{\gamma^u}(s,t)=u^{-1}h_\gamma(s,t)$
\item if $S=T$, $\sel=\altsel$, and the pairing 
$S\times T\map{}\cO(1)$ is symmetric (resp. alternating)
then $h_\gamma$ is alternating (resp. symmetric).
\end{enumerate}
\end{Lem}
\begin{proof} Let $\phi_\gamma:P\map{}\cO$ be the composition 
(\ref{evaluation composition}), so that $$h_\gamma(s,t)=
\phi_\gamma([\eta_\gamma^{-1}(s),\eta_\gamma^{-1}(t)]).$$
The first equality is then immediate from the corresponding property
of the pairing $[\ ,\ ]$.  For the second property, we compute
$$h_{\gamma^u}(s,t)=\phi_{\gamma^u}([\eta_{\gamma^u}^{-1}(s),
\eta_{\gamma^u}^{-1}(t)])=
u^{-2}\phi_{\gamma^u}([\eta_\gamma^{-1}(s),\eta_\gamma^{-1}(t)])$$
and so it suffices to check $\phi_{\gamma^u}=u\phi_\gamma$, which 
is clear from the definition.  For the third property we must show
that $\phi(\mathfrak{p}^\iota)=-\phi(\mathfrak{p})$ 
for $\mathfrak{p}\in P$.  If we write $\mathfrak{p}=\frac{\lambda}
{\gamma^{p^n}-1}$ for some integer $n>0$ and $\lambda\in\Lambda$,
then $$\left(\frac{\lambda}{\gamma^{p^n}-1}\right)^\iota=
\frac{-\lambda^\iota}{\gamma^{p^n}-1}$$ in $P$, and so under the isomorphism
$P\map{\eta_\gamma} \cO_\infty\{1\}\iso \dlim \Lambda_n$
the action of $\iota$ on $P$ becomes  minus the natural action of 
$\iota$ on $\dlim \Lambda_n$.  For $\lambda\in\Lambda_n$, $\ev(\lambda^\iota)=
\ev(\lambda)$ and the claim follows. \end{proof}

Part (b) of the Lemma implies that the pairing of the following theorem
is well defined.

\begin{Thm}\label{second pairing} Keep the assumptions of
Theorem \ref{first pairing} and let $J\subset\Lambda$ be the augmentation
ideal. There is a canonical pairing 
$$h:H^1_\sel(F,S_\infty)\times H^1_\altsel(F,T_\infty)\map{}J/J^2$$
defined by $h(s,t)=h_\gamma(s,t)(\gamma-1)$ where $\gamma$ is any 
topological  generator of $\Gamma$.
This pairing satisfies  $h(\lambda s,t)=h(s,\lambda^\iota t)$ and
the kernels on either side are exactly the universal norms.
If $S=T$, $\sel=\altsel$, and the pairing $T\times T\map{}\cO(1)$ is 
symmetric (resp. alternating) then $h$ is alternating 
(resp. symmetric).
\end{Thm}
\begin{proof} All of the claims follow easily from Lemmas \ref{evaluation}
and \ref{pre-height lemma}, and
the properties of the pairing of Theorem \ref{first pairing}.
\end{proof}

\begin{Rem}\label{twisting} Keeping the notation of the thoerem,
suppose $L$ is a subfield of $F$ with $F_\infty/L$ Galois, and assume
that the action of $G_F$ on $S$ and $T$ extends to an action of $G_L$.
Then for $\bullet=\Iw$, $P$, $K$, or $\infty$, the action of 
$\Lambda$ on $H^1(F,S_\bullet)$ extends to an action of
$\Lambda_L=\cO[[\Gal(F_\infty/L)]]$.  Similarly for every place $v$ of $L$
the action of $\Lambda$ on the semi-localization
$$\bigoplus_{w|v}H^1(F_w,S_\bullet)$$ extends to an action of $\Lambda_L$.
If we assume that the local conditions $\sel$ and $\altsel$ are 
stable under the action of $\Lambda_L$, then $\Lambda_L$ acts on the
associated Selmer groups.
The action of $\Gal(F/L)$ on $\Gamma$ determines a character
$$\omega:\Gal(F/L)\map{}\Z_p^\times$$ by 
$\sigma\gamma\sigma^{-1}=\gamma^{\omega(\sigma)}$ for every 
$\sigma\in\Gal(F_\infty/L)$ and $\gamma\in\Gamma$.
Then for $\sigma\in\Gal(F_\infty/L)$ it can be shown that 
$h(s^\sigma, t^\sigma )= \omega(\sigma)\cdot h(s,t).$
\end{Rem}


\section{derived heights and derivatives of $L$-functions}
\label{derivative}


Throughout this section we work with fixed objects
$S$ and $T$ of $\Mod_\Sigma(\cO)$, and assume that these modules are
in Cartier duality.
Let $\sel$ and $\altsel$ be Selmer structures on $S_K$ and $T_K$, 
respectively, which are everywhere
exact orthogonal complements under the pairing of Lemma \ref{convolution},
and let $$h:H^1_\sel(F,S_\infty)\times H^1_\altsel(F,T_\infty)\map{}J/J^2$$
be the height pairing of Theorem \ref{second pairing}.
We abbreviate
$$Y_S=H^1_\sel(F,S_\infty)\hspace{1cm} Y_T=H^1_\altsel(F,T_\infty).$$

Let $Y=Y_S$ or $Y_T$. For any $r\ge 1$ set
$\delta_r(Y)=Y[J^{r}]/Y[J^{r-1}].$ A choice of topological generator
$\gamma\in\Gamma$ determines an injection $$\phi_{r,\gamma}:
\delta_r(Y)\map{}Y[J]$$ given by $\phi_{r,\gamma}(y)=(\gamma-1)^{r-1}y$, and
we denote its image by $Y^{(r)}\subset Y[J]$. This image
is independent of the choice of $\gamma$ and defines a decreasing
filtration $$\ldots\subset Y^{(3)}\subset 
Y^{(2)}\subset Y^{(1)}=Y[J].$$  The intersection
$\cap Y^{(r)}$ consists of the elements of $Y[J]$ which are universal norms 
in $Y$.

\begin{Rem}
We define a Selmer structure on $S$ by propagating $\sel$ through the
natural inclusion $S\map{}S_\infty$, and again denote this by $\sel$.
This inclusion induces a surjective map $$H^1_\sel(F,S)\map{} Y_S[J]$$
whose kernel is bounded (by the inflation-restriction sequence)
by the order of the group
$H^1(F_\infty/F, H^0(F_\infty,S))$.
Similar remarks hold for $T$.
\end{Rem}

\begin{Def}
We define the $r^\mathrm{th}$
derived height $$h^{(r)}:Y_S^{(r)}\times Y_T^{(r)}\map{}J^{r}/J^{r+1}$$
by the composition 
$$Y_S^{(r)} \times Y_T^{(r)} \map{\phi^{-1}_{r,\gamma}\times \mathrm{id}}
\delta_r(Y_S)\times Y_T^{(r)}\map{h}J/J^2\map{(\gamma-1)^{r-1}}
J^{r}/J^{r+1}.$$ It is easily checked that this is independent 
of the choice of $\gamma$.
Note that $h^{(1)}$ is nothing more than the restriction of $h$ to 
$Y_S[J]\times Y_T[J]$.
\end{Def}

If $S=T$ and $\sel=\altsel$, then for any $x,y\in Y_S$,
$$h^{(r)}(x,y)=\left\{\begin{array}{ll}
 (-1)^r h^{(r)}(y,x)&\mathrm{if\ }e \mathrm{\ symmetric}\\
(-1)^{r+1} h^{(r)}(y,x))&\mathrm{if\ }e \mathrm{\ alternating}.
\end{array}\right.$$
The following lemma implies that the kernels of 
$h^{(r)}$ on the left and right are $Y_S^{(r+1)}$ and 
$Y_T^{(r+1)}$, respectively.  
\begin{Lem}\label{restricted kernels}
If $\lambda_0, \lambda_1\in\Lambda$ are distinguished then
the kernels on the left and right of the restriction of $h$ to
$Y_S[\lambda_0]\times Y_T[\lambda_1]$ are the images of 
$$Y_S[\lambda_0\lambda_1^\iota]\map{\lambda_1^\iota}Y_S[\lambda_0]
\hspace{1cm}Y_T[\lambda_0^\iota\lambda_1]\map{\lambda_0^\iota}
Y_T[\lambda_1]$$ respectively.
\end{Lem}
\begin{proof} Let $Z_S$ and $Z_T$ be the quotients of $Y_S$ and $Y_T$ by
the submodules of universal norms.  The natural map
$$Y_S[\lambda_0]/\lambda_1^\iota Y_S[\lambda_1^\iota\lambda_0]\map{}
Z_S[\lambda_0]/\lambda_1^\iota Z_S[\lambda_1^\iota\lambda_0]
\map{} Z_S/\lambda_1^\iota Z_S$$ is an injection, and the height 
pairing defines an injection
$$Z_S/\lambda_1^\iota Z_S\iso \Hom(Z_T[\lambda_1],J/J^2)\map{}
\Hom(Y_T[\lambda_1],J/J^2).$$ This shows that
$$Y_S[\lambda_0\lambda_1^\iota]\map{\lambda_1^\iota}Y_S[\lambda_0]
\map{}\Hom(Y_T[\lambda_1],J/J^2)$$ is exact.
The kernel on the right is computed similarly.
\end{proof}

The pairing $e_\Iw$ of Lemma \ref{convolution} induces a perfect pairing
$$e_P:S_\Iw\times T_P\map{}P(1).$$ The identification
$T_P\iso T_\infty$ of Lemma \ref{polar isomorphism} and the map 
(\ref{evaluation composition}), both of which depend on a choice of 
topological generator, induce a perfect pairing
$$e_\infty:S_\Iw\times T_\infty\map{}\cO(1)$$ which does not depend on
the choice of generator.  If one views $S_\Iw$ as the space of
$S$-valued measures on $\Gamma$, and $T_\infty$ as the spaces of locally
constant $T$-valued functions on $\Gamma$, then this pairing is integration.
Using Lemma \ref{evaluation}, it can be checked that 
the local conditions $\sel$ and $\altsel$ on $S_\Iw$ and $T_\infty$
are everwhere exact orthogonal complements under the local Tate pairing 
$$H^1(F_v,S_\Iw)\times H^1(F_v,T_\infty)\map{}\cO$$ induced by $e_\infty$.

Denote by $\sel^\rel$ the Selmer structure on $S_\Iw$ whose local conditions
at places of $F$ not dividing $p$ are the same as those of $\sel$, but with
no condition imposed at primes above $p$.  At any place $v$ of $F$ we denote 
$$H^1_\msel(F_v,S_\Iw)=H^1(F_v,S_\Iw)/H^1_\sel(F_v,S_\Iw),$$
and we let 
$$H^1_{\bullet}(F_p,\ \ )=\bigoplus_{v\mid p}H^1_{\bullet}(F_v,\ \ )$$
denote the semi-local cohomology at $p$, where $\bullet$ is either
$\sel$ or $\msel$. Using the fact that the local
condition $\sel$ on $S_\Iw$ is propagated from a local condition on $S_K$,
it is easy to see that $H^1_\msel(F_v,S_\Iw)[f]=0$ for any place $v$ of $F$
and any distinguished $f\in\Lambda$.

For the motivation behind the following definition, see 
\cite{perrin-riou-asterisque}, \cite{rubin-height}, or
\cite{rubin-modular}.

\begin{Def}\label{L function}
For  any element
$$z=\{z_n\}\in\mil H^1_{\sel^\rel}(F,S_n)\iso H^1_{\sel^\rel}(F,S_\Iw)$$
define the $p$-adic $L$-function of $z$, $\pl_z$, to be the image 
of $z$ in $$H^1_\msel(F_p,S_\Iw)\iso \Hom( H^1_\altsel(F_p,T_\infty),\cO).$$
\end{Def}

Define the order of vanishing of $\pl_z$, $\ord(\pl_z)$, to be the
largest power of $J$ by which $\pl_z$ is divisible in $H^1_\msel(F_p,S_\Iw)$.
Equivalently, $\ord(\pl_z)$ is the largest integer $r$ such that
$$\pl_z( H^1_\altsel(F_p,T_\infty)[J^r] )=0.$$

For any topological generator $\gamma\in\Gamma$ and any
$r\le \ord(\pl_z)$ we define $\Der^r_\gamma(\pl_z)$ to be the 
preimage of $\pl_z$ under the injection 
$$H^1_\msel(F_p,S_\Iw)\map{(\gamma-1)^r}H^1_\msel(F_p,S_\Iw).$$
Define $$\pl^{(r)}_z: H^1_\altsel(F_p,T_\infty)\map{}J^r/J^{r+1}$$
by $\pl^{(r)}_z(c)=\Der_\gamma^r(\pl_z)(c)\cdot(\gamma-1)^r$. Then
$\pl^{(r)}_z$ is independent of the choice of $\gamma$.
The restriction of $\pl_z^{(r)}$ to $H^1_\altsel(F_p,T_\infty)[J]$
should be thought of as the ``special value'' of the $r^\mathrm{th}$
derivative of $\pl_z$, and we denote it by $\lambda_z^{(r)}$.

The following is a higher derivative version of Theorem 1 of
\cite{rubin-height}.  Note that the theorem asserts that a local 
divisibility (of the $p$-adic $L$-function $\pl_z$ 
by a power of $J$) implies a global divisibility 
(of $z_0$ by a power of $J$).

\begin{Thm}\label{main theorem}
Keep notation as above, with $r\le\ord(\pl_z)$, and propagate the
Selmer structure $\sel$ to $S$ via $S_\Iw\map{}S$.  This
is equal to the Selmer structure obtained by propagation through
$S\map{}S_\infty$. Then
\begin{enumerate}
\item  $\lambda^{(0)}_z=0$ if and only if $z_0\in H^1_\sel(F,S)$,
\item $\lambda^{(r)}_z=0$ if and only if $r<\ord(\pl_z)$,
\item suppose $0<r\le \ord(\pl_z)$,
then $z_0\in Y_S^{(r)}$ and for any $c\in Y_T^{(r)}$
$$h^{(r)}(z_0,c )=\lambda^{(r)}_z(c_p),$$ 
where $c_p$ is the image of $c$ in $H^1_\altsel(F_p,T_\infty)$.
\end{enumerate}
\end{Thm}

\begin{proof} Fix a topological generator $\gamma\in\Gamma$.
We have $\lambda^{(0)}_z=0$ if and only if 
$\pl_z$ is divisible by $J$.  The first claim now follows from exactness of
$$H^1_\msel(F_p,S_\Iw)\map{\gamma-1}H^1_\msel(F_p,S_\Iw)\map{}
H^1_\msel(F_p,S).$$

For the second,  $r<\ord(\pl_z)$ if and only if $\Der^r_\gamma(\pl_z)$ 
is divisible by $\gamma-1$ in $ H^1_\msel(F_p,S_\Iw)$, and this
is equivalent, by local duality, to $\Der^r_\gamma(\pl_z) $ vanishing on the
$J$-torsion in $H^1_\altsel(F,T_\infty)$.

For the third claim, first suppose $z_0\in Y_S^{(r)}$ 
and fix $c\in Y_T^{(r)}$.  Let $\tilde{z}\in
H^1_{\sel^\rel}(F,S_K)$ be defined by $\tilde{z}=z\otimes(\gamma-1)^{-r}$
and let $y$ denote the image of $\tilde{z}$ in $H^1_{\sel^\rel}(F,S_P)$.
Under the map $$H^1(F,S_P)\map{\eta_\gamma}H^1(F,S_\infty)$$ 
of Lemma \ref{polar isomorphism}, $(\gamma-1)^{r-1}y$ maps to $z_0$.
Since we are assuming that $z_0\in Y_S^{(r)}$ and $c\in Y_T^{(r)}$,
 we may choose $s\in H^1_\sel(F,S_P)$ and $d\in H^1_\altsel(F,T_P)$
which satisfy $$(\gamma-1)^{r-1}\eta_\gamma(s)
=z_0 \hspace{1cm} (\gamma^{-1}-1)^{r-1}\eta_\gamma(d)=c.$$

We have $(\gamma-1)^{r-1}\eta_\gamma(s-y)=0$, and so
\begin{eqnarray*}h^{(r)}(z_0,c)&=&(\gamma-1)^{r-1}\cdot h( \eta_\gamma(s), c)\\
&=&(\gamma-1)^{r-1}\cdot h( \eta_\gamma(y),c).
\end{eqnarray*}  Unraveling the definition of $h$, we find
$$h( \eta_\gamma(y),c)= (\gamma-1)\cdot \phi_\gamma\big([y,
(\gamma^{-1}-1)^{r-1}d]\big)$$ where $\phi_\gamma$ is the composition 
(\ref{evaluation composition}) and $[\ ,\ ]$ is the pairing of Theorem
\ref{first pairing}.
Choose a lift, $\tilde{d}_p$, of $\loc_p(d)$ to $H^1_\altsel(F_p,T_K)$, and
let $\tilde{c}_p=(\gamma^{-1}-1)^{r-1}\tilde{d}_p$.
We now have \begin{eqnarray*}\phi_\gamma\big([y,(\gamma^{-1}-1)^{r-1}d]
\big)&=&\phi_\gamma(\inv_p\big((\gamma-1)^{r-1}y\cup \tilde{d}_p\big)\big)\\
&=&\phi_\gamma\big(\inv_p\big(\tilde{z}\cup 
(\gamma^{-1}-1)^{-r}\tilde{c}_p\big)\big)\\
&=&\Der_\gamma^r(\pl_z)(c_p).
\end{eqnarray*}
Combining all of this gives
$$h^{(r)}(z_0,c)=(\gamma-1)^r\cdot \Der^r_\gamma(\pl_z)(c_p)
=\lambda^{(r)}_z(c_p).$$

We now show by induction on $r$ that $z_0\in Y_S^{(r)}$ for $1\le r\le
\ord(\pl_z)$.  The case $r=1$ follows from parts (a) and (b): since
$0<\ord(\pl_z)$ we must have $z_0\in H^1_\sel(F,S_\infty)[J]$.  For the
inductive step, if $z_0\in Y_S^{(r-1)}$ then we have shown that
$$h^{(r-1)}(z_0,c)=\lambda^{(r-1)}_z(c_p)$$ for every $c\in Y_T^{(r-1)}$.
But since $r-1<\ord(\pl_z)$, part (b) of the proposition implies
that $\lambda^{(r-1)}_z(c_p)=0$ and we conclude that $z_0$ is in the kernel
on the left of $h^{(r-1)}$.  This kernel is exactly $Y_S^{(r)}$, and the
claim is proven. \end{proof}


\section{Selmer groups of ordinary abelian varieties}
\label{abelian varieties}

Let $A$ be an abelian variety defined over $F$ and 
let $\Adual$ be the dual abelian variety. We assume throughout 
that $A$ has good ordinary reduction at all primes of $F$ above $p$, 
and that the primes of bad reduction are finitely 
decomposed in $F_\infty$.  Assume further that $p$ does not ramify 
in $F$, but that all primes of $F$ above $p$ do ramify in $F_\infty$. 
We wish to prove the following:
for each power $p^k$ of $p$
there are generalized Selmer groups
$$H^1_\sel(F_\infty,A[p^k])\subset H^1(F_\infty,A[p^k])\hspace{1cm}
H^1_\seldual(F_\infty,\Adual[p^k])\subset H^1(F_\infty,\Adual[p^k])$$
such that the inclusion
$A[p^k]\hookrightarrow A[p^\infty]$ induces a map of $\Lambda$-modules
$$H^1_\sel(F_\infty,A[p^k])\map{}\Sel_{p^\infty}(A/F_\infty)[p^k]$$
whose kernel and cokernel are finite and bounded as $k$ varies
(and similarly for $\Adual$), where $\Sel_{p^\infty}(A/F_\infty)$
is the usual $p$-power Selmer group associated to $A$.  These generalized
Selmer groups are of the type described in Section \ref{construction},
and so there is a height pairing
$$h_k:H^1_\sel(F_\infty,A[p^k])\times H^1_\seldual(F_\infty,\Adual[p^k])
\map{}J_k/J_k^2$$ where $J_k$ denotes the augmentation ideal of 
$(\Z/p^k\Z)[[\Gamma]]$, and this pairing enjoys all the properties of
that of Theorem \ref{second pairing}.

For every place of $F$, fix once and for all an extension to $\bar{F}$.
At any place $v$ of $F$ above $p$ and for any $k$ 
we let $\Fil_v A[p^k]$ be the kernel of the reduction map
$$A(\bar{F}_v)[p^k]\map{}\tilde{A}[p^k]$$ where 
$\tilde{A}$ is the reduction of 
$A$ at $v$.  Define $\Fil_v \Adual[p^k]$ similarly. Define
$\gr_v A[p^k]$ by exactness of
\begin{equation}\label{filtration sequence}
0\map{}\Fil_v A[p^k]\map{}A(\bar{F}_v)[p^k]\map{}
\gr_v A[p^k]\map{}0\end{equation}
and similarly for $\Adual$.  The reduction map 
on $p^k$-torsion is surjective, and so $\gr_v A[p^k]\iso \tilde{A}[p^k]$.

\begin{Lem}\label{splitting}
The submodules $\Fil_v A[p^k]$ and $\Fil_v \Adual[p^k]$ are exact
orthogonal complements under the Weil pairing.
\end{Lem}
\begin{proof} 
The assumption that
$A$ has ordinary reduction ensures that $\Fil_v A[p^k]$ and 
$\Fil_v\Adual [p^k]$ have exact order $p^{kg}$, 
where $g=\mathrm{dim}(A)$. As modules for the inertia group $\mathcal{I}_v$
of $v$, each is isomorphic to a product of copies of $\mu_{p^k}$, and there
are nonontrivial $\mathcal{I}_v$ invariant pairings 
$\mu_{p^k}^g\times \mu_{p^k}^g\map{}\mu_{p^k}$.
\end{proof}

We set $\cO=\Z/p^k\Z$, $S=A[p^k]$, and $T=\Adual[p^k]$ and use the notation of
the first section.  Shapiro's lemma and Lemma \ref{polar isomorphism}
allow us to identify \begin{equation}\label{abelian shapiro}
H^1(F_\infty,A[p^k])\iso H^1(F,S_\infty)\iso H^1(F,S_P).\end{equation}
For ${\bullet}=\Iw$, $K$, $P$, or $\infty$, and
any place $v$ of $F$ above $p$, the submodule $\Fil_vS\subset S$ 
induces a submodule $\Fil_vS_{\bullet}\subset S_{\bullet}$ in an obvious way, 
and similarly with $S$ replaced by $T$ or with $\Fil_v$ replaced by
$\gr_v$.

Following Coates and Greenberg \cite{coates-greenberg}, we make the 

\begin{Def}\label{ordinary definition}
We define a Selmer structure, $\sel$, on $S_K$
by setting
$$
H^1_\sel(F_v,S_K)=\left\{\begin{array}{ll}
H^1_\unr(F_v,S_K)&\mathrm{if\ }v\not|\ p\\
\mathrm{image}\big(
H^1(F_v,\Fil_v S_K)\map{}
H^1(F_v,S_K)\big)
&\mathrm{if\ }v\mid p\end{array}\right.
$$
and define $\seldual$ on $T_K$ similarly.  
\end{Def}

The local conditions $\sel$ and $\seldual$ are everywhere exact 
orthogonal complements under the local Tate pairing induced by the Weil
pairing and Lemma \ref{convolution}.
We use the Selmer group $H^1_\sel(F,S_P)$ and
the identification (\ref{abelian shapiro}) to define a Selmer group
$H^1_\sel(F_\infty,A[p^k])$,
and make the definition for $\Adual$ similarly.

We must compare these generalized Selmer groups with the usual definitions.
Let $$\Fil_v \bS= \dlim (\Fil_v S_\infty) \hspace{1cm}\bS=\dlim S_\infty$$ 
where the limits are over $k$.  Shapiro's lemma identifies 
$H^1(F,\bS)\iso H^1(F_\infty,A[p^\infty])$ and for any place $v$ of $F$
$$H^1(F_v,\bS)\iso \dlim \bigoplus_w H^1(F_{n,w},A[p^\infty])$$
where the sum is over places $w$ of $F_n$ lying above $v$.

\begin{Def}
Define the ordinary Selmer structure on $\bS[p^k]$ by
$$H^1_\ord(F_v,\bS[p^k])=
\left\{\begin{array}{ll}
H^1_\unr(F_v,\bS[p^k])&\mathrm{else} \\
\mathrm{image}\big(
H^1(F_v,\Fil_v\bS[p^k]) \map{}
H^1(F_v,\bS[p^k])
\big)
&\mathrm{if\ }v\mid p
\end{array}\right.
$$
and let $H^1_\ord(F,\bS)=\dlim H^1_\ord(F,\bS[p^k])$.
\end{Def}

\begin{Prop}\label{greenberg}
The isomorphism $H^1(F,\bS)\iso H^1(F_\infty,A[p^\infty])$ identifies
$$H^1_\ord(F,\bS)\iso\Sel_{p^\infty}(A/F_\infty).$$
\end{Prop}
\begin{proof} Let $\Sigma$ be the set of primes 
of $F$ containing all archimedean
primes, primes above $p$, and primes at which $A$ has bad reduction, 
and let $F_\Sigma$ be the maximal extension of $F$ unramified outside
$\Sigma$.  Then both Selmer groups are defined as the subgroup
of $H^1(F_\Sigma/F_\infty,A[p^\infty])$ of elements which are locally
trivial at every $v\in\Sigma$ not dividing $p$ (this follows from Lemma 
\ref{unramified propagation} and Proposition 1.6.8 of \cite{rubin}),
and are in the kernel of reduction
$$H^1(F_{\infty,w},A[p^\infty])\map{}H^1(F_{\infty,w},\tilde{A}[p^\infty])$$
at places above $p$.
This description of the image of the Kummer map for $w\mid p$
can be found in \cite{coates-greenberg}, in particular Proposition 4.3. 
\end{proof}

\begin{Prop}\label{control}
There are natural maps
$$H^1_\sel(F_\infty,A[p^k])\map{}H^1_\ord(F,\bS[p^k])
\map{}H^1_\ord(F,\bS)[p^k]$$
whose kernels and cokernels are finite and bounded as $k$ varies.
\end{Prop}
\begin{proof} For the first arrow, let $S=A[p^k]$, $\cO=\Z/p^k\Z$,
and recall the identifications $\bS[p^k]\iso S_\infty\iso S_P$.  
The first arrow becomes the inclusion \begin{equation}
\label{control inclusion}H^1_\sel(F,S_P)\subset H^1_\ord(F,S_P).
\end{equation}
The local conditions defining both Selmer groups are
unramified (using Lemma \ref{unramified propagation} for $\sel$) away
from $p$.  Above $p$, the local condition $H^1_\sel(F_v,S_P)$ 
is defined as the image of the composition
$$H^1(F_v,\Fil_v S_K)\map{}H^1(F_v,\Fil_v S_P)\map{}H^1(F_v,S_P),$$
while the ordinary local condition is defined as the image of the second arrow.
It follows that we have injections
\begin{eqnarray*}
H^1_\ord(F,S_P)/H^1_\sel(F,S_P)&\map{}&\bigoplus_v
H^1(F_v,\Fil_v S_P)/H^1(F_v,\Fil_v S_K)\\
&\map{}&\bigoplus_v H^2(F_v,\Fil_v S_\Iw)\end{eqnarray*} 
where the sums are over the primes $v$ of $F$ above $p$.
By local duality and Shapiro's lemma the order of 
$H^2(F_v,\Fil_v S_\Iw)$ is bounded by
$$\bigoplus_v H^0(F_v,\gr_v T_\infty)\iso 
\bigoplus_{w|p}\tilde{A}^\vee(L_w)[p^k]$$
where $T=\Adual[p^k]$, the second sum is over all primes of $F_\infty$ 
above $p$, $L_w$ denotes
the residue field of $F_\infty$ at $w$, and $\tilde{A}^\vee$ is the 
reduction of $\Adual$ at $w$.
This group is finite and bounded as $k$ varies, by the assumption 
that all primes of $F$ above $p$ ramify  in $F_\infty$.

To control the kernel and cokernel of 
\begin{equation}\label{control map}
H^1_\ord(F,\bS[p^k])\map{}H^1_\ord(F,\bS)[p^k]\end{equation}
 we use the exact sequence
$$0\map{}A(F_\infty)[p^\infty]/p^kA(F_\infty)[p^\infty]\map{}
H^1(F,\bS[p^k])\map{}H^1(F,\bS)[p^k]\map{}0$$
The kernel of (\ref{control map}) is bounded by the order of 
$A(F_\infty)[p^\infty]$ modulo 
its maximal divisible subgroup.  By the snake lemma, to bound the 
cokernel of (\ref{control map}) it suffices to bound the kernel
of $$\bigoplus_v H^1(F_v,\bS[p^k])/H^1_\ord(F_v,\bS[p^k])\map{}
\bigoplus_v H^1(F_v,\bS)/H^1_\ord(F_v,\bS)$$ and we compute this kernel
term by term in three cases.

Suppose first that $v$ does not divide $p$, and that $A$ has good reduction
at $v$.  Then the kernel of \begin{equation}\label{control kernels}
H^1(F_v,\bS[p^k])/H^1_\ord(F_v,\bS[p^k])\map{}H^1(F_v,\bS)/H^1_\ord(F_v,\bS)
\end{equation} injects into the kernel of
$H^1(F_v^\unr,\bS[p^k])\map{}H^1(F_v^\unr,\bS)$ which is
$$H^0(F_v^\unr,\bS)/p^k H^0(F_v^\unr,\bS)\iso \bS/p^k\bS =0.$$
If $v$ does not divide $p$ and $A$ has bad reduction at $v$, 
then our assumption that $v$ is finitely
decomposed in $F_\infty$ implies that $$H^1_\ord(F_v,\bS)=
H^1_\ord(F_v,\bS[p^k])=0$$ by Lemma \ref{unramified propagation}, and
so we must bound the kernel of the map
$$H^1(F_v,\bS[p^k])\map{}H^1(F_v,\bS).$$  This kernel
is $H^0(F_v,\bS)/p^k H^0(F_v,\bS)$ which  has order bounded by the 
size of the quotient of $\bigoplus A(F_{\infty,w})[p^\infty]$ by its maximal
divisible subgroup, where the sum is over primes $w$ of 
$F_\infty$ above $v$.

Lastly, if $v\mid p$ it suffices to bound the kernel of 
$$H^1(F_v,\gr_v \bS[p^k])\map{} H^1(F_v,\gr_v\bS).$$  This kernel is 
controlled by the order of
$$H^0(F_v,\gr_v\bS)\iso \bigoplus_w \tilde{A}(L_w)[p^\infty]$$
modulo its maximal divisible subgroup,
where the sum is over primes of $F_\infty$ above $p$, $L_w$ is the
residue field $F_\infty$ at $w$, and $\tilde{A}$ is the reduction of
$A$ at $v$. \end{proof}


\section{Semi-simplicity of Iwasawa modules}
\label{semi-simple}


We now set $\Lambda=\Z_p[[\Gal(F_\infty/F)]]$, 
let $J$ be the augmentation ideal of $\Lambda$, and
denote by $I_n\subset\Lambda$ the
kernel of the natural projection
$\Lambda\map{}\Z_p[\Gal(F_n/F)]$, so that $I_0=J$.
Keep $A_{/F}$ as in the preceeding section, so that
$A$ has good ordinary reduction at all primes of $F$ above $p$, 
the primes of bad reduction are finitely 
decomposed in $F_\infty$, $p$ does not ramify 
in $F$, and all primes of $F$ above $p$ do ramify in $F_\infty$. 

Let $Y_k=H^1_\sel(F_\infty, A[p^k])$ be the generalized 
Selmer group of Section \ref{abelian varieties}, 
and set $Y_k(F_n)=Y_k[I_n]$. We denote by
$$\ldots\subset Y^{(3)}_k\subset 
Y^{(2)}_k\subset Y^{(1)}_k=Y_k(F)$$
the filtration of Section \ref{derivative}.  The 
surjection $A[p^{k+1}]\map{p}A[p^k]$
induces a map $Y_{k+1}^{(r)}\map{}Y_{k}^{(r)}$, 
and we define 
$$Y_\infty^{(r)}=\mil Y_k^{(r)}\hspace{1cm}
Y_\infty(F_n)=\mil Y_k(F_n)$$
so that $Y_\infty^{(r)}$ defines a 
decreasing filtration of $Y_\infty(F)$.
Set $$Y=\Sel_{p^\infty}(A/F_\infty)
\hspace{1cm}X=\Hom_{\Z_p}(Y,\Q_p/\Z_p).$$
These are cofinitely and finitely generated $\Lambda$-modules,
respectively.
By Propositions \ref{greenberg} and \ref{control}
we have canonical maps $Y_k\map{}Y[p^k]$
with kernel and cokernel finite and bounded as $k$ varies,
and these induce maps
$$Y_k^{(r)}\iso Y_k[J^r]/Y_k[J^{r-1}]\map{}
(Y[J^r]/Y[J^{r-1}])[p^k]
$$
whose kernels and cokernels
are again finite and bounded as $k$ varies.
The $\Z_p$-module $Y[J^r]/Y[J^{r-1}]$ is cofinitely generated,
and so
\begin{eqnarray}\label{kernel rank}
\rank (Y_\infty^{(r)})&=&\rank\left(\mil (Y[J^r]/Y[J^{r-1}])[p^k]\right)\\
&=&\corank (Y[J^r]/Y[J^{r-1}])\nonumber \\
&=&\rank (J^{r-1}X/J^rX)\nonumber .
\end{eqnarray}

\begin{Lem}\label{semi-integral}
Define $S_p(A/F_n)=\mil\Sel_{p^k}(A/F_n)$.
There is a canonical isomorphism
$$Y_\infty(F_n)\otimes\Q_p\map{j}S_p(A/F_n)\otimes\Q_p.$$
Furthermore this isomorphism is \emph{semi-integral} in the sense
that there are integers $t_0, t_1$, independent of $n$, such that
$p^{t_0}j(Y_\infty(F_n))$ is contained  the lattice generated by
$S_p(A/F_n)$, and $p^{t_1}j^{-1}(S_p(A/F_n))$ is contained in the 
lattice generated by $Y_\infty(F_n)$.
\end{Lem}
\begin{proof}
We have maps 
$$Y_k(F_n)\rightarrow\Sel_{p^\infty}(A/F_\infty)[I_n+p^k\Lambda]
\leftarrow \Sel_{p^\infty}(A/F_n)[p^k]$$
whose kernels and cokernels are finite and bounded as both $k$ and $n$ vary
(the second arrow by Mazur's control theorem \cite{mazur-abelian}), 
and so taking the inverse
limit in $k$ and tensoring with $\Q_p$ we obtain a semi-integral 
isomorphism \begin{equation}\label{integral}
Y_\infty(F_n)\otimes\Q_p\iso (\mil\Sel_{p^\infty}(A/F_n)[p^k])\otimes\Q_p.
\end{equation}
For every $k$ there is a canonical surjection
\begin{equation}\label{strong integral}
\Sel_{p^k}(A/F_n)\map{}\Sel_{p^\infty}(A/F_n)[p^k]
\end{equation}
and these maps are compatible, in the obvious sense, as $k$ varies.
As $k$ varies the kernels are uniformly bounded by the order of the 
finite group $A(F_n)[p^\infty]$, and so passing to the inverse limit 
over $k$ and tensoring with $\Q_p$, we see that 
the right hand side of (\ref{integral})
is isomorphic to $S_p(A/F_n)\otimes \Q_p$.  Surjectivity of
(\ref{strong integral}) implies that this isomorphism 
identifies the lattices generated by
$\mil\Sel_{p^\infty}(A/F_n)[p^k]$ and  $S_p(A/F_n)$.
\end{proof}

The subspace of $S_p(A/F)\otimes\Q_p$ 
generated by the image of $Y_\infty^{(r)}$ 
under the isomorphism 
\begin{equation}\label{comparison}
Y_\infty(F)\otimes\Q_p\map{}S_p(A/F)\otimes\Q_p
\end{equation}
will be denoted 
$S^{(r)}_p(A/F)$, and we set $S^{(\infty)}_p(A/F)=\cap_r
S^{(r)}_p(A/F)$.
 Let $W_r$ be the one-dimensional $\Q_p$-vector
space $$W_r=(J^r/J^{r+1})\otimes \Q_p.$$  The derived height pairings of 
Section \ref{derivative} are compatible as $k$ varies, and
passage to the limit yields the pairing of the following theorem.

\begin{Thm}\label{vector height}
There is a filtration 
$$\ldots\subset S_p^{(3)}(A/F)\subset S_p^{(2)}(A/F)\subset
 S_p^{(1)}(A/F)=S_p(A/F)\otimes\Q_p$$
such that $\mathrm{dim}_{\Q_p} S_p^{(r)}(A/F)=\rank(J^{r-1}X/J^r X)$
(and similarly for $A^\vee$), and a sequence of pairings
$$h^{(r)}:S_p^{(r)}(A/F)\times S_p^{(r)}(A^\vee/F)\map{}W_r$$
such that the kernel on the left (resp. right) is $S_p^{(r+1)}(A/F)$
(resp. $S_p^{(r+1)}(A^\vee/F)$).  

The subspace $S_p^{(\infty)}(A/F)$ is the subspace of \emph{universal
norms} in the usual sense.  That is, $S_p^{(\infty)}(A/F)$ is the
subspace generated by the intersection over $n$ of the image
of corestriction $S_p(A/F_n)\map{}S_p(A/F)$.
Furthermore, if $\phi:A\map{}A^\vee$ is a polarization 
then $h^{(r)}$ satisfies
$$h^{(r)}(a,\phi(b))=(-1)^{r+1}h^{(r)}(b,\phi(a))$$ 
for all $a,b\in S_p^{(r)}(A/F)$.
\end{Thm}
\begin{proof}
All of the claims are immediate from the corresponding properties
of the derived heights over discrete coefficient rings, 
together with the equality (\ref{kernel rank}), except for the
characterization of $S_p^{(\infty)}(A/F)$.  Suppose 
$x\in S_p^{(\infty)}(A/F)$.  Then for some integer $t$ there is
a $y\in Y_\infty(F)$, contained in  $Y_\infty^{(r)}$ for every $r$,
such that $y$ maps to $p^t x$ under (\ref{comparison}).
Fix a topological generator $\gamma\in\Gamma$ and set 
$g_n=\frac{\gamma^{p^n}-1}{\gamma-1}\in\Lambda$.
If $y_k$ denotes the image of $y$ in $Y_k(F)$, 
we claim that $y_k$ is in the image of $g_n:Y_k(F_n)\map{}Y_k(F)$
for every $n$.  Indeed, $y_k\in (\gamma-1)^r Y_k$ for every 
$r$, and it follows from Lemma \ref{coterminal} that $y_k$
is a universal norm (in the sense of Definition \ref{norm def})
in $Y_k$.  In particular $y_k$ is divisible by every $g_n$.
Passing to the limit, we must have that $y$ is in the image
of $g_n:Y_\infty(F_n)\map{}Y_\infty(F)$ for every $n$, say $y=g_n z_n$.
If $t_0$ is as in Lemma \ref{semi-integral}, then the image of
$p^{t_0} z_n$ in $S_p(A/F_n)\otimes\Q_p$ is integral and corestricts
to $p^{t+t_0}x$.  Hence $p^{t+t_0}x$ is a universal norm.
The opposite implication is entirely similar.
\end{proof}

By the structure theorem for finitely-generated Iwasawa modules,
we may fix a pseudo-isomorphism of $\Lambda$-modules
$$X\sim \Lambda^{e_\infty}\oplus M\oplus M'$$
such that  $M'$ is a torsion $\Lambda$-module with characteristic ideal
prime to $J$, and $M$ has the form
$$M\iso (\Lambda/J)^{e_1}\oplus (\Lambda/J^2)^{e_2}\oplus\ldots.$$

The first statement of the following is due to Perrin-Riou
\cite{perrin-riou}.
\begin{Cor}\label{semi-simplicity}
The integers $e_i$ satisfy the following properties:
\begin{enumerate}
\item the height pairing $h^{(1)}$ is nondegenerate if and only if
$e_i=0$ for $1<i\le\infty$,
\item
$e_\infty=\mathrm{dim}_{\Q_p}S_p^{(\infty)}(A/F)$,
\item $e_r=\dim_{\Q_p}\big(S_p^{(r)}(A/F)/S_p^{(r+1)}(A/F)\big),$
\item $e_{r}\equiv 0\pmod{2}$ when $r$ is even.
\end{enumerate}
\end{Cor}
\begin{proof}
By Theorem \ref{vector height} we have the equality
$$\mathrm{dim}_{\Q_p}S_p^{(r)}(A/F)=\rank_{\Z_p} (J^{r-1}X/J^rX)
=e_r+e_{r+1}+e_{r+2}+\cdots +e_\infty
$$
which proves all but the final claim.
A choice of polarization of $A$ determines an isomorphism
$$S_p^{(r)}(A/F)\iso S_p^{(r)}(A^\vee/F)$$
and the induced height pairing on $S_p^{(r)}(A/F)$ is alternating
when $r$ is even, by the last part of Theorem \ref{vector height}.
This implies that $S_p^{(r)}(A/F)/S_p^{(r+1)}(A/F)$ is even dimensional.
\end{proof}

In the case where $F_\infty$ is the cylotomic $\Z_p$-extension,
it is conjectured that $h^{(1)}$
is nondegenerate.  When $F=\Q$ and $A$ is modular
it is known by the work of Kato that $e_\infty=0$
(i.e. $X$ is a torsion module), but it is not known that $e_i=0$ for
$i>1$.

Now suppose that $F$ is a quadratic imaginary field, $F_\infty$ is
the anticyclotomic $\Z_p$-extension, and $A=E\times_\Q F$ for
some  elliptic curve $E_{/\Q}$
satisfying the ``Heegner hypothesis'' that all primes of bad
reduction are split in $F$ (which, in particular, implies
our hypothesis that the primes of bad reduction of $A$ are finitely
decomposed in $F_\infty$).  In this situation it is known by 
the work of Bertolini \cite{bertolini} and Cornut \cite{cornut}
 that $e_\infty=1$, hence
$h^{(r)}$ is degenerate for every $r$.  The next best thing one could hope
for is that $S_p^{(2)}(A/F)$ is one dimensional, hence equal to 
$S_p^{(\infty)}(A/F)$,
but this is still too optimistic.
By Remark \ref{twisting}, $h^{(1)}$ satisfies
$$h^{(1)}(x^\tau,y^\tau)=-h^{(1)}(x,y),$$ where $\tau$ is
complex conjugation.  This forces the plus and minus eigencomponents
of $S_p^{(1)}(A/F)$ under $\tau$ to be self-orthogonal, and so if 
$$s^+=\dim_{\Q_p}S_p^{(1)}(A/F)^+\hspace{1cm}
s^-=\dim_{\Q_p}S_p^{(1)}(A/F)^-,$$
the kernel of $h^{(1)}$ has dimension at least $|s^+-s^-|$.

\begin{Conj}(Bertolini-Darmon, Mazur)
In the situation above, 
the dimension of $S_p^{(2)}(A/F)$ is $|s^+-s^-|$,
and the dimension of $S_p^{(3)}(A/F)$ is $1$.
\end{Conj}

Assuming the conjecture, Corollary \ref{semi-simplicity} implies that
$e_2=|s^+-s^-|-1$.  Mazur's control theorem gives
$$s^++s^-=\dim_{\Q_p}S^{(1)}_p(A/F)=1+e_1+e_2$$ 
and so
$$X\sim \Lambda\oplus(\Lambda/J)^{e_1}\oplus(\Lambda/J^2)^{e_2}\oplus M'$$
with $e_1=2\min\{s^+,s^-\}$ and $M'$ having
characteristic ideal prime to $J$.

Returning to the general case, we wish to reformulate Theorem
\ref{main theorem} in the present setting.  This is merely an
exercise in passing from results on $A[p^k]$ to 
results on the Tate module $T_p(A)$, although some caution 
is needed, owing to our slightly strange choice of Selmer structures
on $A[p^k]$.  Let $\Sigma$ be the
set of places of $F$ consisting of all archimedean places and all
primes at which $A$ has bad reduction.  Define
\begin{eqnarray*}
H^1_\Iw(F_\Sigma/F_\infty,T_p(A))&=&\mil H^1(F_\Sigma/F_n,T_p(A))\\
Z_\infty&=&\mil \oplus_{v|p}H^1(F_{n,v},T_p(A))\\
Z_{\infty,f}&=&\mil \oplus_{v|p}E(F_{n,v})\otimes\Z_p\\
Z_{\infty,s}&=&Z_\infty/Z_{\infty,f}
\end{eqnarray*}
where we regard $Z_{\infty,f}$ as a $\Lambda$-submodule of $Z_\infty$
via the Kummer map.
Suppose we are given some $z\in H^1_\Iw(F_\Sigma/F_\infty,T_p(A))$
whose image $\pl_z\in Z_{\infty,s}$ actually lands in
$J^rZ_{\infty,s}$ with $r>0$, say $\pl_z=(\gamma-1)^r y$
for some choice of topological generator $\gamma\in\Gamma$
and some $y\in Z_{\infty,s}$. Let $y_0$ denote the image of $y$ in 
$H^1(F_p,T_p(A))/(A(F_p)\otimes\Z_p)$
and define 
$$\lambda^{(r)}_z: A^\vee(F_p)\otimes\Q_p\map{} W_r$$
by $\lambda^{(r)}_z(a)=(y_0,a)\cdot(\gamma-1)^r$, where $(\ ,\ )$ is
the local Tate pairing.

\begin{Thm}\label{height formula}
With notation and definitions as above, let $z_0$ be the image of
$z$ in $H^1(F,T_p(A))$.  Then $z_0\in S_p^{(r)}(A/F)$
and for any $c\in S_p^{(r)}(A^\vee/F)$ we have
$$h^{(r)}(z_0,c)=\lambda^{(r)}_z(c_p)$$
where $c_p$ is the image of $c$ in $A^\vee(F_p)\otimes\Q_p$.
\end{Thm}
\begin{proof}
Fix some positive integer $k$.
First note that our assumption that all finite primes of $\Sigma$ are
finitely decomposed in $F_\infty$ implies that $z$ is unramified away from 
$p$, by Corollary B.3.5 of \cite{rubin}.  Hence, if $z(k)$ denotes the
image of $z$ in $\mil H^1(F_n,A[p^k])$ (the limit being over $n$), we have
$z(k)\in H^1_{\sel^\rel}(F,S_\Iw)$ where $S=A[p^k]$, $\sel$ is the Selmer
structure obtained by propagating the Selmer structure of Definition
\ref{ordinary definition} through the injection $S_\Iw\map{}S_K$,
and the notation $\sel^\rel$ has the same meaning as in Section 
\ref{derivative}.  We claim that the image of $z(k)$
in  $H^1_{/\sel}(F,S_\Iw)$, which we shall denote $\pl_z(k)$, is divisible
by $(\gamma-1)^r$. Indeed, it suffices to show that 
the natural map $Z_\infty\map{} H^1(F_p,S_\Iw)$ takes $Z_{\infty,f}$
into $H^1_\sel(F_p,S_\Iw)$, so that we have a well-defined map
$Z_{\infty,s}\map{}H^1_{/\sel}(F,S_\Iw).$
 By Tate local duality, this is equivalent to the condition
 that the natural map 
$$H^1(F_p,T_\infty)\map{}\dlim \oplus_{v|p}H^1(F_{n,v},A^\vee[p^\infty])$$
takes $H^1_{\sel^\vee}(F_p,T_\infty)$ into the image of $\dlim\oplus_{v|p}
(A(F_{n,v})\otimes \Q_p/\Z_p)$ under the Kummer map, 
where $T=A^\vee[p^k]$ and $\sel^\vee$ is the
Selmer structure of Definition \ref{ordinary definition}.  This claim
now follows by replacing $A$ by $A^\vee$ in Propositions 
\ref{greenberg} and \ref{control} (strictly speaking, these
propositions state  that there is a natural map on global Selmer groups
$H^1_{\sel^\vee}(F,T_\infty)\map{}\Sel_{p^\infty}(A^\vee/F_\infty)$,
but the proofs proceed by showing that the isomorphisms of local cohomology 
induced by Shapiro's lemma take the local conditions
defining the left hand side into the local conditions defining the
right hand side).
We have now shown that the order of vanishing of $\pl_z(k)$ is at least $r$,
and so appealing to Theorem \ref{main theorem} we see that $z_0(k)$,
the image of $z(k)$ under
$$H^1(F,S_\Iw)\map{}H^1(F,A[p^k])\map{}H^1(F_\infty,A[p^k])[J],$$ lies in 
the submodule $Y_k^{(r)}$ defined at the beginning of this section.
Passing to the limit over $k$ we have $z_0\in Y_\infty^{(r)}$,
and therefore $z_0\in S_p^{(r)}(A/F)$.

Now fix some $c\in Y^{\vee{(r)}}_\infty$
(the module defined in the same way as $Y_\infty^{(r)}$, but with
$A$ and $\sel$ replaced by $A^\vee$ and $\sel^\vee$).
Passing to the limit in Theorem \ref{main theorem}
we have the equality
$h^{(r)}(z_0, c)=\lambda_z^{(r)}(c_p)$ in $J^r/J^{r+1}$, 
and extending by $\Q_p$-linearity
proves the same result for $c\in S_p^{(r)}(A^\vee/F)$.
\end{proof}

In the special case where $S_p^{(1)}(A/F)$ is one-dimensional and
$S_p^{(\infty)}(A/F)$ is trivial,
the above theorem may be regarded as a $p$-adic, 
cohomological formulation of the 
Birch and Swinnerton-Dyer conjecture.  
In this situation $\exists$ $\delta\ge 1$ such that
$$\dim_{\Q_p}S_p^{(r)}(A/F)=\left\{\begin{array}{ll}
1 & r\le\delta \\
0 & r>\delta.\end{array}\right.$$
If $z_0\not=0$, the left hand side of the equality of the 
theorem can then be interpreted 
as a regulator term.  By the theorem, $\pl_z$ must vanish to order $\delta$,
and the right hand side is like the value of the
$\delta^\mathrm{th}$-derivative.  
This interpretation breaks down when the dimension of
$S_p^{(1)}(A/F)$ is greater than one, since the formula only allows one
to compute the heights of elements against some fixed element $z_0$,
and not all pairwise heights in a basis.

Of course this is only interesting if one can construct an element $z$
for which we are somehow justified in calling 
$\pl_z$ a $p$-adic $L$-function, and
there are essentially two cases where this is understood: when $A/\Q$ is an 
elliptic curve, Kato has constructed such an element for which $\pl_z$
is related to the Mazur--Swinnerton-Dyer $p$-adic $L$-function of $E$,
and when $A$ is a CM elliptic curve the Euler system of elliptic
units is related, by work of Yager, to the two-variable $p$-adic $L$-function
constructed by Katz.  For applications of 
Theorem \ref{height formula} in these special cases, 
see \cite{agboola} for the elliptic unit
case and \cite{rubin-height, rubin-modular} for the case of Kato's Euler 
system.


\begin{thebibliography}{10}

\bibitem{agboola}
A. ~Agboola and B. ~Howard.
\newblock Anticyclotomic Iwasawa theory of CM elliptic curves.
\newblock {\em preprint}, 2003.

\bibitem{bertolini}
M.~Bertolini.
\newblock Selmer groups and {H}eegner points in anticyclotomic
  $\mathbf{Z}_p$-extensions.
\newblock {\em Compositio Mathematica}, 99:153--182, 1995.

\bibitem{derived-heights}
M.~Bertolini and H.~Darmon.
\newblock Derived $p$-adic heights.
\newblock {\em American Journal of Mathematics}, 117(6):1517--1554, 1995.

\bibitem{coates-greenberg}
J.~Coates and R.~Greenberg.
\newblock Kummer theory for abelian varieties over local fields.
\newblock {\em Invent. Math}, 124:129--174, 1996.

\bibitem{cornut}
C. ~Cornut.
\newblock Mazur's conjecture on higher Heegner points.
\newblock {\em Invent. Math}, 148:495--523, 2002.

\bibitem{flach}
M.~Flach.
\newblock A generalisation of the {C}assels-{T}ate pairing.
\newblock {\em J. Reine Angew. Math.}, 412:113--127, 1990.

\bibitem{greenberg-representations}
R.~Greenberg.
\newblock {I}wasawa theory for $p$-adic representations.
\newblock In {\em Algebraic Number Theory}, volume~17 of {\em Adv. stud. in
  pure math.} Princeton University Press, 1989.

\bibitem{mazur-abelian}
B.~Mazur.
\newblock Rational points of abelian varieties with values in towers of number
  fields.
\newblock {\em Inventiones Math.}, 18:183--266, 1972.

\bibitem{mazur-rubin}
B.~Mazur and K.~Rubin.
\newblock Kolyvagin systems.
\newblock {\em preprint}, 2002.


\bibitem{mazur-tate}
B.~Mazur and J.~Tate.
\newblock Canonical height pairings via biextensions.
\newblock In M.~Artin and J.~Tate, editors, {\em Arithmetic and geometry},
  volume~1, pages 195--238. Birkh\"{a}user, Boston, 1983.

\bibitem{nek-selmer}
J.~Nekov\'{a}\v{r}.
\newblock Selmer complexes.
\newblock Unpublished manuscript.

\bibitem{perrin-riou}
B.~Perrin-Riou.
\newblock Th\'{e}orie d'{I}wasawa et hauteurs $p$-adiques.
\newblock {\em Invent. Math.}, 109:137--185, 1992.

\bibitem{perrin-riou-asterisque}
B.~Perrin-Riou.
\newblock {\em $p$-adic $L$-functions and $p$-adic Representations}.
\newblock American Mathematical Society, 2000.

\bibitem{rubin-height}
K.~Rubin.
\newblock Abelian varieties, $p$-adic heights and derivatives.
\newblock In {\em Algebra and Number Theory}. Walter de Gruyter and Co., 1994.

\bibitem{rubin-modular}
K.~Rubin.
\newblock Euler systems and modular elliptic curves.
\newblock In {\em Galois Representations in Arithmetic Algebraic Geometry},
  London Math. Soc. Lecture Notes 254, pages 351--367. Cambridge Univ. Press,
  1998.

\bibitem{rubin}
K.~Rubin.
\newblock {\em Euler Systems}.
\newblock Princeton University Press, 2000.

\end{thebibliography}
\end{document}